\documentclass[12pt,a4paper]{article}
\usepackage{tikz}
\usepackage[T1]{fontenc}
\usepackage[british]{babel}
\usepackage[utf8]{inputenc}
\usepackage{amsmath,amssymb,amsthm}
\usepackage{booktabs}
\usepackage{array}
\usepackage{amsmath, amssymb}
\usepackage{float}
\usepackage{comment}
\usepackage{rotating} 
\usepackage{longtable} 
\usepackage{hyperref}
\usepackage{mathrsfs}
\usepackage{mathabx}

\newcommand\Ss{\mathcal{S}}

\newcommand\XX{\mathfrak{X}}

\DeclareMathOperator{\N}{N}

\DeclareMathOperator{\C}{C}

\DeclareMathOperator{\Z}{Z}

\DeclareMathOperator{\Ker}{Ker}

\DeclareMathOperator{\Syl}{Syl}
\DeclareMathOperator{\Sylow}{Sylow}
\DeclareMathOperator{\Aut}{Aut}

\DeclareMathOperator{\Id}{Id}

\DeclareMathOperator{\Stab}{Stab}

\DeclareMathOperator{\GL}{GL}

\DeclareMathOperator{\Qq}{Q}

\DeclareMathOperator{\Irr}{Irr}

\newtheorem{theorem}{Theorem}
\newtheorem{lem}[theorem]{Lemma}
\newtheorem{athm}{Theorem}

\theoremstyle{definition}

\newtheorem{corollary}[theorem]{Corollary}

\newtheorem*{remark*}{Remark}
\newtheorem{question}{Question}

\begin{document}
\title{Connectivity of $p$-subgroup posets with irreducible characters}
\author{Hangyang Meng\thanks{E-mail: hymeng2009@shu.edu.cn. Department of Mathematics 
and Newtouch Center for Mathematics of Shanghai University,
Shanghai 200444, P. R. China. This research is sponsored by Natural Science Foundation of Shanghai 
(24ZR1422800) and National Natural Science Foundation of China (12471018). }~~and~~Yuting Yang\thanks{E-mail: yutingyang@shu.edu.cn. Department of Mathematics 
and Newtouch Center for Mathematics of Shanghai University,
Shanghai 200444, P. R. China.} 
}
\date{}
\maketitle
\begin{abstract}
Let $G$ be a finite group. For a prime $p$ and an integer $e \geq 0$, we denote by $\Gamma_{p,e}(G)$ the set of all pairs $(H, \varphi)$, where $H$ is a $p$-subgroup of $G$ of order greater than $p^e$ and $\varphi$ is a complex irreducible character of $H$. In this paper, we investigate the connected components of the poset $\Gamma_{p,e}(G)$. For the case $e = 0$, we prove that $\Gamma_{p,0}(G)$ is disconnected if and only if either $G$ has a strongly $p$-embedded subgroup, or every Sylow $p$-subgroup of $G$ contains a unique subgroup of order $p$. Furthermore, for $e = 1$ and $G$ a $p$-group, we show that the number of connected components of $\Gamma_{p,1}(G)$ equals the order of the intersection of all subgroups of $G$ of order $p^2$.
\end{abstract}

\section{Introduction}
\vskip 10pt

Let $G$ be a finite group and $p$ be a prime. For an integer $e \geq 0$, we denote by
$$\Ss_{p,e}(G)=\{H \leq G \mid H~\text{is a $p$-group with order}>p^e\}$$
the set of all $p$-subgroups of $G$ with order greater than $p^e$. In particular, we write $\Ss_p(G)=\Ss_{p,0}(G)$ for $e=0$, which is the set of all non-trivial $p$-subgroups of $G$.  Some algebraic topological properties of the poset $\Ss_p(G)$ under the inclusion relation is first introduced by Brown~\cite{Brown1975} and studied further by Quillen~\cite{Quillen1978}, for example, the connected components of $\Ss_p(G)$.

Recall that, for a poset \(X\), two elements are connected or lie in the same \emph{connected component} if and only if they can be joined by a finite sequence in which every consecutive pair is comparable. We write \(\pi_0 X\) for the set of all connected components of \(X\). Quillen~\cite[Proposition~5.2]{Quillen1978} show that the connected components of the poset $\Ss_p(G)$ are related to \emph{strongly $p$-embedded subgroups} of $G$. More general, the connected components of $\Ss_{p,e}(G)$ are presented due to Brown~\cite[Proposition~7]{Brown2000}.

Regarding results for other subgroup posets, Lucido~\cite{Lucido2003} classified all finite groups such that the poset \(L(G)\) of all non-trivial proper subgroups of \(G\) is disconnected. Han and Zheng~\cite{HanZheng2025} further studied the contractibility of \(L(G)\) as a finite space.
Gu, Meng and Guo~\cite{GuMengGuo2025} introduced the poset $\Gamma(G)$ of all non-trivial subgroups of $G$ associated with irreducible characters defined by
 $$\Gamma(G)=\{(H,\varphi) \mid 1 \neq H \leq G, \varphi \in \Irr(H)\},$$
 where $\Irr (H)$ is the set of all complex irreducible characters of $H$, and it is proven~\cite[Theorem~A]{GuMengGuo2025} that $\Gamma(G)$ is disconnected if and only if $G$ is a $p$-group possessing only one subgroup of order $p$ for some prime $p$, that is, $G$ is a cyclic $p$-group or isomorphic to a quaternion group $\Qq_{2^n}$. (see~\cite[Theorem 8.2 of Chapter 3]{Huppert2025}
 
 Motivated by these results above, we will study the following poset $\Gamma_{p,e}(G)$ for a finite group $G$, a prime $p$ and an integer $e \geq 0$, which is defined by
\[\Gamma_{p,e}(G)=\{(H,\varphi)| H \in \Ss_{p,e}(G), \varphi \in \Irr(H)\}
\]
with the partial order "$\leq$" on \(\Gamma_{p,e}(G)\):
\[
(H, \varphi) \leq (K, \psi)~~\text{if}~~H \leq K \text{ and } [\varphi, \psi_H] \neq 0,
\]
where \([ \cdot, \cdot ]\) denotes the inner product of characters. 

Our first result indicates that, in some sense, we can reduce the connectivity of $\Gamma_{p,e}(G)$ to studying the connectivity of $\Gamma_{p,e}(P)$ for a Sylow $p$-subgroup $P$ of $G$.
\begin{athm}\label{thm-A}
Let $G$ be a finite group and $P$ be a $\Sylow$ $p$-subgroup of $G$ for some prime $p$. Fix an integer $e \geq 0$.  
Then the number of connected components of $\Gamma_{p,e}(G)$ satisfies
$$|\pi_0\Gamma_{p,e}(G)|=|\pi_0\XX(P)|\cdot |\pi_0\Ss_{p,e}(G)| \leq |\pi_0\Gamma_{p,e}(P)|\cdot |\pi_0\Ss_{p,e}(G)|,$$
where $\XX(P)=\{(H,\varphi) \in \Gamma_{p,e}(G) \mid H,P~\text{are connected in}~\Ss_{p,e}(G)\}$.
In particular, suppose that $P$ has a subgroup isomorphic to one of the following groups:
\begin{itemize}
\item[\emph{(1)}] $\C_{p^{e+1}} \times \C_{p^{e+1}}$, where $\C_{p^{e+1}}$ is cyclic of order $p^{e+1}$; or
\item[\emph{(2)}] an elementary Abelian $p$-subgroup of order $p^{e+2}$.
\end{itemize}
Then $\Gamma_{p,e}(P)$ is connected and $|\pi_0\Gamma_{p,e}(G)|=|\pi_0\Ss_{p,e}(G)|=|G:\Stab_G([P])|$, where $\Stab_G([P])$ is the $G$-stabilizer of the connected component $[P]$ containing $P$ of $\Ss_{p,e}(G)$. 
\end{athm}
Note that a proper subgroup $H$ is said to be a \emph{strongly $p^{e}$-embedded subgroup} of $G$ if $p^e$ divides $|H|$ and $p^e \nmid |H \cap H^x|$ for each $x \in G-H$. For $e=1$, $H$ is so-called a strongly $p$-embedded subgroup of $G$. (See~\cite[46.4]{Aschbacher2000}) It is known \cite[Proposition~7]{Brown2000} that
$\Ss_{p,e}(G)$ is disconnected i.e. $|\pi_0\Ss_{p,e}(G)|>1$ if and only if the $G$-stabilizer $\Stab_G([P])$ is a strongly $p^{e+1}$-embedded subgroup of $G$. We will give a proof for completeness in Section $2$.

We write $\Gamma_p(G)=\Gamma_{p,0}(G)$ for $e=0$. We will give a description on finite groups when 
$\Gamma_{p}(G)$ is disconnected.
\begin{athm}\label{thm-B}
Let $G$ be a finite group and $P$ be a $\Sylow$ $p$-subgroup of $G$. Then $\Gamma_p(G)$ is disconnected if and only if $P$ has only one subgroup of order $p$ or $G$ has a strongly $p$-embedded proper subgroup. Furthermore, if $P$ has the unique subgroup of order $p$, then $\Gamma_p(G)$ has exactly $p|G:\N_G(\Omega_1(P))|$ connected components.
\end{athm}

It is worth pointing out that if $P$ has at least two different subgroups of order $p$, then $P$ has a subgroup isomorphic to $\C_p \times \C_p$. It easily follows from Theorem~\ref{thm-A} that 
$\Gamma_p(P)$ is connected and $|\pi_0\Gamma_{p}(G)|=|\pi_0\Ss_{p}(G)|$.

It will be next interesting to deal with $\Gamma_{p,e}(G)$ for the case $e \geq 1$ but the question becomes more complicated. 
For $e=1$, we provide a characterization of finite $p$-groups $G$ for which $\Gamma_{p,1}(G)$ fails to be connected.
\begin{athm}\label{thm-C}
Let $G$ be a $p$-group with $|G| \geq p^2$ and let $I$ be the intersection of all $p$-subgroups of $G$ with order $p^2$. Then $|\pi_0\Gamma_{p,1}(G)|=|I|$. In particular, $\Gamma_{p,1}(G)$ is disconnected if and only if $I \neq 1$.
\end{athm}
From now on, all groups considered will be finite.

\section{Connected components of $\Gamma_{p,e}(G)$}
Recall that, in a finite poset $(X,\leq)$,  we say there is a path from $x$ to $y$ (written by $x \sim y$) for $x,y \in X$ if there exists $x_0,x_1, \cdots, x_n \in X$ such that $x=x_0,~x_n=y$ and either $x_i \leq x_{i+1}$ or $x_i \geq x_{i+1}$ for each $i=0,1,\cdots,n$. Denote by
$$[x]=\{ y \in X \mid y \sim x\}$$
the connected component containing $x$ of $X$ and by
$\pi_0(X)=\{[x] \mid x \in X\}$
the set of all connected components of $X$. In particular, $X$ is called connected if $X$ has only one connected component; otherwise $X$ is called disconnected.

First we consider the action of $G$ on the set of all connected components $\pi_0 \Ss_{p,e}(G)$ of $\Ss_{p,e}(G)$, which is defined by
$$[H]^g:=[H^g],~H \in S_{p,e}(G),~g \in G.$$
By Sylow's Theorem,  every connected components of $S_{p,e}(G)$ contains at least one Sylow $p$-subgroup $P$. Hence, $G$ acts transitively on $\pi_0\Ss_{p,e}(G)$, so we have $|\pi_0\Ss_{p,e}(G)|=|G:N|$, where $N=\Stab_G([P])$ is the stabilizer of some components $[P]$. The following lemma
due to Brown~\cite[Proposition 7]{Brown2000} will show that such stabilizer $N$ is a strongly $p^{e+1}$-embedded subgroup of $G$.  We give a proof here for completeness.

\begin{lem}\label{lem:brown-coset}
Let $G$ be a group and $p$ be a prime. Let $e\geq 0$ be an integer with $p^{e+1}$ dividing $|G|$.
The following conditions on a proper subgroup $M$ of $G$ are equivalent:
\begin{itemize}
\item[\emph{(1)}] $M$ contains the stabilizer of a component of $\Ss_{p,e}(G)$;
\item[\emph{(2)}] For some $\Sylow$ $p$-subgroup $S$ of $G$, $M$ contains $\N_G(P)$ for all $P\in \Ss_{p,e}(S)$;
\item[\emph{(3)}] $p^{e+1}\mid |M|$, and $M$ contains $\N_G(P)$ for all $P\in \mathcal{S}_{p,e}(M)$;
\item[\emph{(4)}] $M$ contains $\N_G(S)$ for some $\Sylow$ $p$-subgroup $S$ of $G$, and for each $P\in \mathcal{S}_{p,e}(M)$, $M$ contains all $p$-subgroups $\geq P$;
\item[\emph{(5)}] $M$ is a strongly $p^{e+1}$-embedded subgroup of $G$.
\end{itemize}
\end{lem}
\begin{proof}
$(1) \Rightarrow (2)$ Take a component $X$ of $\Ss_{p,e}(G)$ with $\Stab_G(X) \subseteq M$. Let $S$ be a $\Sylow$ $p$-subgroup of $G$ contained in $X$, for each $P\in \Ss_{p,e}(G)$ with $P\leq S$ and $g\in \N_G(P)$, we know that $P^g=P$. As $P\in X$ and $P=P^g\in X^g$,  we have  $P\in X \cap X^g$. Note that $X^g$ is also a connected component and different connected components are disjoint. Hence $X^g=X$, which implies that $g\in \Stab_G(X) \subseteq M$.

$(2) \Rightarrow (3)$ Clearly, $S\leq \N_G(S) \leq M$. As $S$ is a $\Sylow$ $p$-subgroup of $G$ and $|S| \geq p^{e+1}$, so $p^{e+1} \mid |M|$. For each $P\in \Ss_{p,e}(M)$, there exists an element $g$ of $M$ such that $P^g \leq S$, hence $\N_G(P^g) \subseteq M$ by hypothesis. Clearly, $\N_G(P)^g=\N_G(P^g) \subseteq M$. Hence $\N_G(P)\subseteq M$.

$(3) \Rightarrow (4)$ Let $S$ be a Sylow $p$-subgroup of $M$ and $|S|\geq p^{e+1}$. If $|S|<|G|_p$, by~\cite[Chapter 3, Main Theorem 2.3(c)]{Huppert2025}, $|S|<|\N_G(S)|_p$. But, as $S \in \Ss_{p,e}(M)$, it follows from $(3)$ that $\N_G(S) \leq M$ and 
$|S|<|\N_G(S)|_p \leq |M|_p$, which contradicts that $S \in \Syl_p(M)$. Hence $|S|=|G|_p$ and $S \in \Syl_p(G)$. By $(3)$ again,  $\N_G(S) \subseteq M$. For each $P\in \Ss_{p,e}(M)$, if there exists a $p$-subgroup $Q$ and $Q\nleq M$ such that $P\leq Q$, then $Q>Q \cap M \geq P \geq 1$, one has
$$Q \cap M < \N_Q(Q\cap M)=Q\cap \N_G(Q\cap M) \subseteq Q\cap M,$$
which is a contradiction. Then $Q \leq M$, as desired. 

$(4) \Rightarrow (5)$ Clearly $p^{e+1}$ divides $|M|$. Now we prove that for each $x\notin M$, $p^{e+1} \nmid |M \cap M^x|$. Assume that there exists an element $x\notin M$ such that $p^{e+1}$ divides $|M \cap M^x|$. Let $Q$ be the $p$-subgroup of $|M \cap M^x|$ with $|Q| \geq p^{e+1}$, then there exists a $\Sylow$ $p$-subgroup $H$ of $M$ such that $Q \leq H \leq M$. As $Q^{x^{-1}} \leq M^{x^{-1}} \cap M$ and $Q^{x^{-1}} \leq H^{x^{-1}}$, then $H^{x^{-1}} \cap M \geq Q^{x^{-1}}$.
As $Q^{x^{-1}} \in \Ss_{p,e}(M)$, by $(4)$, $H^{x^{-1}} \leq M$. Hence $S, H$ and $H^{x^{-1}}$ are $\Sylow$ $p$-subgroups of $M$, one has $S^m=H$ and $S^{m'}=H^{x^{-1}}$ with $m, m'$ in $M$, then $S^{m'm^{-1}x}=S$, hence $m'm^{-1}x\in \N_G(S)\leq M$, so $x\in M$.

$(5) \Rightarrow (1)$  By definition, for each $x \notin M$, $p$ divides $|M: M \cap M^x|$.  We know that
$$|G:M|=\sum_{MxM} |M:M \cap M^x| \equiv 1\pmod{p}.$$
which implies $M$ contains a $\Sylow$ $p$-subgroup $P$ of $|P|\geq p^{e+1}$. Let $X$ be the component containing $P$, for each $g$ in $\Stab_G(X)$, $P^g \in X$, that is $P$ and $P^g$ are in the same component. Hence there exists a sequence of $\Sylow$ groups $P=P_0, P_1,..,P_k=P^g$ such that $P_{i-1} \cap P_i \geq p^{e+1}$. Let $x_i$ be such that $(P_{i-1})^{x_i}=P_i$ for $1\leq i \leq k$, we can choose $x_k=x_{k-1}^{-1} \cdots x_1^{-1} \cdot g$ such that $P_{k-1}^{x_k}=P^g$, so $g=x_1 \cdots x_k$. Assume $P_{i-1} \leq M$, then $p^{e+1} \leq P_{i-1} \cap P_i \leq M \cap M^{x_i}$, which shows that $p^{e+1} $ divides $ |M:M \cap M^{x_i}|$.  By $(5)$,  we have $x_i \in M$ and $P_i \leq  M$. Thus by induction $x_i \in M$ and $P_i \leq M$ for all $i$. Hence $g\in M$, that is, $\Stab_G(X) \leq M$.
\end{proof}

As a direct corollary, we have:
\begin{corollary}\label{cor-disconnected}
Let $G$ be a group and $p$ be a prime. Fix an integer $e \geq 0$. Then $\Ss_{p,e}(G)$ is disconnected if and only if $G$ has a strongly $p^{e+1}$-embedded subgroup of $G$.      
\end{corollary}

We have to recall irreducible characters of the direct product of two groups, for example, see~\cite[Chapter 19]{James2001} or \cite[4.20 Definition]{Isaacs1976}.
Let $G=AB$ be the direct product of two subgroups $A$ and $B$. It is known that every irreducible character of $G$ can be presented as $(\varphi \times \psi)(g)=\varphi(a)\psi(b)$ for $g=ab \in G, a \in A$ and $b \in B$, where $\varphi$ and $\psi$ are irreducible characters of $A$ and $B$ respectively. Note that The restriction of $\varphi \times \psi$ to $A$, denoted by $(\varphi \times \psi)_A$, is exactly $\varphi$.


\begin{lem}\label{lem-direct-product}
Fix a prime $p$ and an integer $e \geq 0$.
Let $G=T_1T_2\cdots T_s$ be the direct product of some $p$-subgroups $T_i$ of $G$, where $s \geq 2$ and $1\leq i \leq s$. If $|G:T_i|\geq p^{e+1}$ for each $1 \leq i \leq s$ , then $\Gamma_{p,e}(G)$ is connected.
\end{lem}
\begin{proof}
For each $1\leq i \leq s$, we write 
$$M_i=\prod_{1 \leq j \neq i \leq s} T_j.$$
Clearly $G=M_i \times T_i$ and $|M_i|=|G: T_i| \geq p^{e+1}$ for $i=1,2,\cdots,s$. 
For each $\chi \in \Irr(G)$, write $\chi=\chi_0$ and define by induction $$\chi_{i+1}=(\chi_i)_{M_{i+1}}\times 1_{T_{i+1}} \in \Irr(G)~\text{for}~i=0,1,2,s-1,$$
where $1_{T_{i+1}}$ is the principle character of $T_{i+1}$. As $(\chi_{i+1})_{M_{i+1}}=(\chi_i)_{M_{i+1}}$ by definition, we easily see that
$(G,\chi_i) \geq (M_{i+1}, (\chi_i)_{M_{i+1}}) \leq (G,\chi_{i+1})$ in $\Gamma_{p,e}(G)$ for $i=0,1,2,\cdots, s-1$. It implies that there is a long path in $\Gamma_{p,e}(G)$:
$$(G,\chi_0)\geq (M_1,(\chi_0)_{M_1})\leq (G,\chi_1) \geq (M_2,(\chi_1)_{M_2})\leq (G,\chi_2)\geq \dots \leq (G,\chi_s).$$

Next we will show that $\chi_s=1_G$, the principle character of $G$. In fact, we will prove by induction on $i$ that $\chi_i(g)=1$ for each $g\in T_1T_2 \cdots T_i$. When $i=1$, $\chi_1(g)=((\chi_0)_{M_1}\times 1_{T_1})(g)=1$ for each $g\in T_1$. 
Assume that $\chi_{i-1}(g)=1$ for each $g\in T_1 \cdots T_{i-1}$. For each $g  \in T_1T_2\cdots T_{i-1}T_i$, we assume that $g=ab$ for $a \in T_1T_2\cdots T_{i-1}$ and $b \in T_i$. By induction, $\chi_{i-1}(a)=1$. Hence $\chi_i(g)=((\chi_{i-1})_{M_i}\times 1_{T_i})(ab)=\chi_{i-1}(a)\cdot 1_{T_i}(b)=1\cdot 1=1$, as desired.  Hence $\chi_s(g)=1$ for all $g \in T_1T_2\cdots T_s=G$, that is, $\chi_s=1_G$, as claimed. It follows that $(G,\chi)$ is connected with $(G,1_G)$ for each $\chi \in \Irr(G)$, that is, $(G, \chi)\sim (G,1_G)$.

For every vertex $(H, \psi) \in \Gamma_{p,e}(G)$,  we can take $\chi \in \Irr(G)$ such that $[\psi^G, \chi] \neq 0$. It implies that $(H,\psi) \leq (G,\chi)$.
Since $(G,\chi) \sim (G, 1_G)$, we have $(H,\psi) \sim (G,1_G)$. Hence $\Gamma_{p,e}(G)$ is connected and the proof is complete.
\end{proof}

\begin{lem}\label{lem-subgroup-connected}
Fix a prime $p$ and an integer $e \geq 0$. Let $G$ be a $p$-group and $H \leq G$ with $|H| \geq p^{e+1}$. Then 
$$\Gamma_{p,e}(G)=\{[(H,\varphi)] \mid \varphi \in \Irr(H)\}.$$
In particular, $|\pi_0\Gamma_{p,e}(G)| \leq |\pi_0\Gamma_{p,e}(H)|$.
\end{lem}
\begin{proof}
For each $(T,\phi) \in \Gamma_{p,e}(G)$,  we can take $\chi \in \Irr(G)$ as a constitute of $\phi^G$ and 
$\varphi \in \Irr(H)$ as a constitute of $\chi_H$. Then
$$(T,\phi) \leq (G,\chi) \geq (H,\varphi)~~\text{in}~~\Gamma_{p,e}(G),$$
which implies that every element $(T,\phi)$ in $\Gamma_{p,e}(G)$ is connected with some $(H,\varphi)$. Hence $\Gamma_{p,e}(G)=\{[(H,\varphi)] \mid \varphi \in \Irr(H)\}$ and  $|\pi_0\Gamma_{p,e}(G)| \leq |\pi_0\Gamma_{p,e}(H)|$, as desired.
\end{proof}

Let $X$ and $Y$ be two posets. A map $f: X \rightarrow Y$ is called a poset map if $f$ is order-preserving, that is, $f(x_1) \leq f(x_2)$ in $Y$ for any $x_1 \leq x_2$ in $X$. 
\begin{lem}\label{lem-component}
Let $f:X \rightarrow Y$ be a poset map between two posets $X$ and $Y$. Let $C \in \pi_0Y$ be a connected component of $Y$. Then $f^{-1}(C)$ is the (disjoint) union of some (possibly zero) connected components of $X$. Moreover,
$$\pi_0X= \bigcup_{C \in \pi_0Y} \pi_0f^{-1}(C).$$
Additionally, if $f$ is surjective, then $|\pi_0X| \geq |\pi_0Y|$.
\end{lem}
\begin{proof}
Clearly, for each $x_1, x_2 \in X$, if $x_1$ and $x_2$ are in the same connected component of $\pi_0X$, then $f(x)$ and $f(y)$ are in the same connected component of $\pi_0Y$. If $C=\varnothing$, then $f^{-1}(C)=\varnothing$. If $C\neq \varnothing$, for each $x \in f^{-1}(C)$, there exists the only connected component $[x] \in \pi_0 X$ such that $f([x]) \subseteq C$, so $[x] \subseteq f^{-1}(C)$.  Hence, 
$$f^{-1}(C)=\bigcup_{x \in f^{-1}(C)}[x].$$
For each $x,y \in f^{-1}(C)$, if $x=y$, then $[x]=[y]$, so $f^{-1}(C)$ is the disjoint union of some connected components of $X$.

Clearly, we know that $X=\bigcup_{C \in \pi_0Y} f^{-1}(C)$. Hence, $$\pi_0X= \bigcup_{C \in \pi_0Y} \pi_0f^{-1}(C).$$ The proof is complete.
\end{proof} 

Now we can prove Theorem~\ref{thm-A}.
\begin{proof}[\textbf{\emph{Proof of Theorem~\ref{thm-A}}}]
If $p^{e+1}$ does not divide $|G|$, then $\Gamma_{p,e}(G)=\Ss_{p,e}(G) =\varnothing$ and $|\pi_0\Gamma_{p,e}(G)|=|\pi_0\Ss_{p,e}(G)|=0$. The result is trivial. Hence we may assume that $p^{e+1}$ divides $|G|$.
Define a map:
$$f: \Gamma_{p,e}(G) \rightarrow \Ss_{p,e}(G);~(H, \psi) \mapsto H,$$
which is a well-defined poset map. Since each connected component of $\Ss_{p,e}(G)$ contains a $\Sylow$ $p$-subgroup of $G$, we may assume that $|\pi_0\Ss_{p,e}(G)|=s$ and 
$$\pi_0\Ss_{p,e}(G)=\{[P_i] \mid P_i \in \Syl_p(G), 1\leq i \leq s\},$$
where we set $P_1=P$. Write 
$$\XX(P_i)=f^{-1}([P_i])=\{(H,\varphi) \in \Gamma_{p,e}(G) \mid H,P_i~\text{are connected in}~\Ss_{p,e}(G)\}.$$
It follows from Lemma~\ref{lem-component} that 
$$\pi_0\Gamma_{p,e}(G)=\bigcup_{i=1}^s \pi_0f^{-1}([P_i])=\bigcup_{i=1}^s \pi_0\XX(P_i).$$
For each $i$, by $\Sylow$'s Theorem, there is an element $g \in G$ such that $P^g=P_i$. We can 
define a poset map $\alpha_g$ via the conjugation by $g$:
$$\alpha_g: \XX(P) \rightarrow \XX(P_i);~(H,\psi) \mapsto (H^g, \psi^g),$$
where $\psi^g(h^g)=\psi(h)$ for $h \in H$ and $\psi^g \in \Irr(H^g)$. Note that $H \sim P$ implies that
$H^g \sim P^g=P_i$. Hence $\alpha_g$ is well-defined. It is not difficult to check that $\alpha_g$ is a poset map, and $\alpha_g$ is bijective because of the existence of the inverse map $\alpha_{g^{-1}}: \XX(P_i) \rightarrow \XX(P)$.  Hence $ \XX(P)$ is isomorphic to $\XX(P_i)$ as a poset, and clearly
$|\pi_0\XX(P)|=|\pi_0\XX(P_i)|$. Hence it implies that
$$|\pi_0\Gamma_{p,e}(G)|=\sum_{i=1}^s |\pi_0\XX(P_i)|=|\pi_0\XX(P)|s=|\pi_0\XX(P)|\cdot|\pi_0\Ss_{p,e}(G)|.$$
We next will show that $|\pi_0\XX(P)| \leq |\pi_0\Gamma_{p,e}(P)|$.
For each $(H, \psi) \in \XX(P)$, by the definition, $H,P$ are connected in $\Ss_{p,e}(G)$, that is, there exist $\Sylow$ $p$-subgroups $Q_0,Q_1,\cdots,Q_n $ of $G$ such that $|Q_{i-1} \cap Q_{i}|>p^e$ for $i=1,2,\cdots,n$ and 
$$H \leq Q_0 \geq Q_0 \cap Q_1 \leq Q_1 \geq Q_1 \cap Q_2 \leq \cdots  \leq Q_{n-1} \geq Q_{n-1} \cap Q_n\leq Q_n=P.$$
We can take some suitable characters $\chi_i \in \Irr(Q_i), i=0,1,\cdots, n$ and $\psi_j \in \Irr(Q_j \cap Q_{j+1}), j=0,1,\cdots, n-1$ such that
$$(H, \psi) \leq (Q_0,\chi_0) \geq (Q_0 \cap Q_1,\psi_0) \leq \cdots $$
$$\cdots \leq (Q_{n-1},\chi_{n-1}) \geq (Q_{n-1} \cap Q_n, \psi_{n-1})\leq (Q_n,\chi_n)=(P,\chi_n).$$
Hence each $(H,\psi)$ is connected with some element $(P,\chi_n) \in \Gamma_{p,e}(P)$. Note that
$\pi_0\Gamma_{p,e}(P)=\{[(P,\chi)] \mid \chi \in \Irr(P)\}$ by Lemma~\ref{lem-subgroup-connected}. Hence $|\pi_0\XX(P)| \leq |\pi_0\Gamma_{p,e}(P)|$ holds.

Now we suppose that $P$ has a subgroup isomorphic to $\C_{p^{e+1}} \times \C_{p^{e+1}}$ or
an elementary Abelian $p$-subgroup of order $p^{e+2}$.  Applying Lemma~\ref{lem-direct-product} and Lemma~\ref{lem-subgroup-connected}, $\Gamma_{p,e}(P)$ is connected, which implies that $|\pi_0\XX(P)|=|\pi_0\Gamma_{p,e}(P)|=1$.
Hence $|\pi_0\Gamma_{p,e}(G)|=|\pi_0\Ss_{p,e}(G)|$. Considering the transitive action of $G$ on $\pi_0\Ss_{p,e}(G)$ via conjugation,  the last part of Theorem~\ref{thm-A} follows. 
\end{proof}

\section{Connected components of $\Gamma_{p,0}(G)$}

A better perspective for counting the number of connected components is to regard the posets as  finite spaces and consider their topological homotopy equivalence.
We have to recall some basic facts on the homotopy theory of finite spaces, see~\cite{barmak2011algebraic,Stong1966325} for more details.

Recall that a finite poset $(X,\leq)$ has an intrinsic topology generated by the basis $\{U_x \mid x \in X\}$, where \(U_x=\{y \mid y \leq x \}\). Moreover, a map $f:X \rightarrow Y$ between posets is order-preserving if and only if it is continuous with the intrinsic topology. Two continuous maps $f,g: X \rightarrow Y$ are homotopic (denote by $f \simeq g$) if and only if there exists a fence of continuous maps $f_0, f_1, \cdots, f_n : X \rightarrow  Y$ such that $f_0 = f , f_n = g$ and either $f_i \leq f_{i+1}$ or $f_i \geq f_{i+1}$ holds for each $0 \leq i < n$. 
Recall that $f_i \leq f_{i+1}$ (resp. $f_i \geq f_{i+1}$) if $f_i(x) \leq f_{i+1}(x)$ (resp. $f_i(x) \geq f_{i+1}(x)$) for each $x \in X$. Two spaces $X,Y$ are called homotopy-equivalent if there exist continuous maps $f:X \rightarrow Y$ and $g: Y \rightarrow X$ such that $g \circ f \simeq \Id_X$ and $f \circ g \simeq \Id_Y$, where $\Id_X$ and $\Id_Y$ are identity maps on $X$ and $Y$ respectively.

In this section, we consider connected components of $\Gamma_{p}(G)=\Gamma_{p,0}(G)$.
\begin{proof}[\textbf{\emph{Proof of Theorem~\ref{thm-B}}}]

Suppose that $\Gamma_p(G)$ is disconnected and $P$ has more than one subgroup of order $p$ and we will show that $G$ has a strongly $p$-embedded subgroup.
Since $P$ has more than one subgroup of order $p$,  $P$ has a subgroup isomorphic to $\C_p \times \C_p$. It follows from Theorem~\ref{thm-A} that $|\pi_0\Gamma_p(G)|=|\pi_0\Ss_p(G)|=|G: \Stab_G(X)|$, where $X \in \pi_0\Ss_p(G)$. As $\Gamma_p(G)$ is disconnected, $|G: \Stab_G(X)|=|\pi_0\Gamma_p(G)|>1$, and hence $\Stab_G(X)$ is a proper subgroup of $G$. 
Applying Lemma~\ref{lem:brown-coset} for $e=0$,  $\Stab_G(X)$ is a strongly $p$-embedded subgroup of $G$, as desired. 

Conversely, if $G$ has a strongly $p$-embedded subgroup, by Corollary~\ref{cor-disconnected}, $|\pi_0\Ss_p(G)|> 1$.  It follows from Theorem~\ref{thm-A} that $$|\pi_0\Gamma_p(G)| \geq |\pi_0\Ss_p(G)|>1.$$ Hence $\Gamma_p(G)$ is disconnected. 
Now we assume that $P$ has only one subgroup of order $p$.
Let $\Ss=\{X_1,X_2,...,X_s\}$ denote all subgroups of $G$ with order $p$ and set $X_1=\Omega_1(P)= \langle x \in P | x^p=1 \rangle$ as the unique subgroup of $P$ with order $p$. As $G$ acts transitively on all $\Sylow$ $p$-subgroups of $G$ and  $\Omega_1(P)^g=\Omega_1(P^g)$ for each $g \in G$, $G$ also acts transitively on $\Ss$. Hence we see that
$$s = |\Ss| = |G: \N_G(X_1)| = |G:\N_G(\Omega_1(P))|.$$ 
Let $\mathscr{X}=\{(X_i,\theta)|1\leq i \leq s,\theta \in \Irr(X_i)\} \subseteq \Gamma_p(G)$ and clearly $\mathscr{X}$, as a finite space, is a discrete space of $ps$ points.  For each $(H,\varphi) \in \Gamma_p(G)$, as
each Sylow $p$-subgroup has only one subgroup of order $p$, $H$ also has  only one subgroup $\Omega_1(H)$ of order $p$.
Let $\theta \in \Irr(\Omega_1(H))$ with $[\theta^H,\varphi] \neq 0$.
Since $\Omega_1(H) \leq \Z(H)$, $\theta$ is $H$-invariant. It follows from~\cite[2.27 Lemma]{Isaacs1976}, we know that $\theta$ is the unique irreducible constituent of $\varphi_{\Omega_1(H)}$. Hence, we can give a well-defined map:
$$g: \Gamma_p(G) \rightarrow \mathscr{X};~ (H,\varphi) \mapsto (\Omega_1(H),\theta).$$
We claim that $g$ is a poset map. In fact, for $(T,\psi) \leq (H,\varphi) \in \Gamma_p(G)$, $\psi$ is an irreducible constituent of $\varphi_T$. Note that the uniqueness of the subgroup of $H$ with order $p$ implies that  $\Omega_1(T)=\Omega_1(H)$.
Since $\theta$ is the unique irreducible constituent of $\varphi_{\Omega_1(H)}$,
$\theta$ is also the unique irreducible constituent of $\psi_{\Omega_1(H)}$, which implies that 
$$g((T,\psi))=(\Omega_1(T),\theta)=(\Omega_1(H),\theta)=g((H,\varphi)),$$ as claimed. Let $i:\mathscr{X} \rightarrow \Gamma_p(G)$ be an inclusion map, also a poset map. Since $g \circ i=\Id_{\mathscr{X}}$ and, for each $(H,\varphi) \in \Gamma_p(G)$,
$$i\circ g((H,\varphi))=i(\Omega_1(H),\theta)\leq(H,\varphi)=\Id_{\Gamma_p(G)}.$$
Thus, $\Gamma_p(G)$ is homotopy-equivalent to $\mathscr{X}$, which is a discrete space of $ps$ points. Hence $|\pi_0\Gamma_p(G)|=ps$. In particular, $\Gamma_p(G)$ is disconnected.
The proof is complete.
\end{proof}

\section{Connected components of $\Gamma_{p,1}(G)$}
In this section, we consider connected components of $\Gamma_{p,1}(G)$. We begin by stating some basic definition and theorems of group theory in the next lemma. 

\begin{lem}\label{basis of p-subgroup}
Let $G$ be a group of order $p^n$ with $n\geq 1$.
\begin{itemize}
\item[\emph{(1)}] If $1 \neq H \unlhd G$ then $H \cap \Z(G) \neq 1$. In particular, $\Z(G)\neq 1$;
\item[\emph{(2)}] For each $0 \leq i \leq n$, $G$ has a normal subgroup of order $p^i$;
\item[\emph{(3)}] If $G$ has only one subgroup with order $p^2$, then $G$ is cyclic or $G \cong \C_p \times \C_p$.
\end{itemize}
\end{lem}
\begin{proof}
See~\cite[Chapter 3, Theorem 7.2(a),(c) and Theorem 8.3]{Huppert2025}.
\end{proof}
\begin{lem}\label{lem-p^2 components}
Let $G$ be a $p$-group for some prime $p$ with $|G| \geq p^2$. Suppose that $G$ is cyclic or $G \cong \C_p \times \C_p$.
Then  $\Gamma_{p,1}(G)$ has $p^2$ connected components.
\end{lem}
\begin{proof}
If $G\cong \C_p\times \C_p$, obviously, $|\pi_0\Gamma_{p,1}(G)|=p^2$. Assume that $G$ is cyclic. Let $N$ be the unique subgroup of $G$ with order $p^2$. Since $G$ is Abelian, for each $\chi \in \Irr(G)$, $\chi$ is linear and $\chi_N \in \Irr(N)$. Hence we can define a map $\alpha$:
$$\alpha: \Gamma_{p,1}(G) \rightarrow \Gamma_p(N);~ (G,\chi) \mapsto(N,\chi_N).$$
For each $(G_1, \chi_1)\leq (G_2, \chi_2)$,  we have $\chi_1=(\chi_2)_{G_1}$ and $N \leq G_1 \leq G_2$, which implies that $\alpha(G_1, \chi_1)=(N,(\chi_1)_N)=(N,(\chi_2)_N)=\alpha(G_2, \chi_2)$. 
Hence $\alpha$ is a poset map. It follows from Lemma~\ref{lem-component} that $|\pi_0\Gamma_{p,1}(G)|\geq |\pi_0\Gamma_{p}(N)|=p^2.$
Note that $|\pi_0\Gamma_{p,1}(G)|=|\{[(N,\theta)]|\theta \in \Irr(N)\}| \leq p^2$. Hence  $|\pi_0\Gamma_{p,1}(G)|=p^2$.
\end{proof}

Next we should deal with the connected components of $\Gamma_{p,1}(G)$ when $G$ is a $p$-group of order $p^3$.  
\begin{lem}\label{lem-non-abelian group of order p^3}
Let $G$ be a non-abelian group of order $p^3$ for some prime $p$ and write $N=\Phi(G)$. Then for each non-principle irreducible character $\theta$ of $N$, there exists an irreducible character $\chi$ of $G$ such that $\theta^G=p \chi$ and $\chi_N=p\theta$. 
\end{lem}
\begin{proof}
Note that $N=\Phi(G)=\Z(G)=G'$ for non-abelian group of order $p^3$.
It follows from \cite[26.6 Theorem]{James2001} that $G$ has exactly $p^2$ linear characters and $p-1$ non-linear irreducible characters of degree $p$.  Let $\chi \in \Irr(G)$ be a character with $[\chi, \theta^G] \neq 0$. Note that $\chi$ is non-linear; otherwise, $N=G' \leq \Ker \chi$, $\chi_N=\lambda 1_N$ for some $\lambda$, which contradicts that $\theta \neq 1_N$ is an irreducible constituent of $\chi_N$.
Hence $\chi$ is non-linear and $\chi(1)=p$.
Since $N=\Z(G)$, by~\cite[2.27 Lemma]{Isaacs1976}, we know that $\chi_N$ is a multiplicity of $\theta$. Comparing the degree of $\chi$ and $\theta$, we have $\chi_N=p\theta$. Note that $[\chi, \theta^G]=[\chi_N,\theta]=p$ and $\theta^G(1)=p^2$. It implies that $\theta^G=p\chi$, as desired.
\end{proof}


\begin{lem}\label{Frattini subgroup of order p}
Let $G$ be a group of order $p^3$ for a prime $p$ with $|\Phi(G)|=p$. Then $\Gamma_{p,1}(G)$ has exactly $p$ connected components. 
\end{lem}

\begin{proof}
Write $N=\Phi(G)$.
For each $(K,\psi)\in \Gamma_{p,1}(G)$, $K=G$ or $K$ is maximal in $G$. Hence $N=\Phi(G) \leq K$. As
$N \leq \Z(G)$, $\psi_N$ is a multiplicity of $\theta$ for some $\theta \in \Irr(N)$.
We can define a map:
$$f \colon \Gamma_{p,1}(G) \rightarrow \Gamma_{p}(N);~~(K,\psi) \mapsto (N, \theta),$$
where $\theta$ is the unique irreducible constituent of $\psi_N$. It is easy to check that $f$ is a surjective poset map. Hence, by Lemma~\ref{lem-component}, $|\pi_0\Gamma_{p,1}(G)| \geq |\pi_0\Gamma_{p}(N)|=p$.

Assume that $G$ is Abelian, then $|\Phi(G)|=p$ implies that $G\cong \C_{p^2} \times \C_p$.  
Write $G=A \times B$ as the direct product of $A \cong \C_{p^2}$ and $B \cong \C_p$. 
Note that $\Gamma_p(N)$ is a discrete space of $p$ points and $\{(N,\theta)\}$ is a connected component of $\Gamma_p(N)$ for each $\theta \in \Irr(N)$.
We will show that $f^{-1}(N,\theta)=\{(K,\psi) \in \Gamma_{p,1}(G) \mid \psi_N=\theta\}$ is exactly one connected component of $\Gamma_{p,1}(G)$. In fact, for $(K_1,\psi_1), (K_1,\psi_2) \in f^{-1}(N,\theta)$, there exist $\chi_1, \chi_2 \in \Irr(G)$ such that $(\chi_i)_{K_i}=\psi_i$ for $i=1,2$.
It implies that $(\chi_i)_N=\theta$ and so $(G,\chi_i) \in f^{-1}(N,\theta)$. Note that $N=\Phi(G)=\Omega_1(A) \leq A$.  Then we will see a chain in $\Gamma_{p,1}(G):$
$$(G,\chi_1) \geq (A,(\chi_1)_A) \leq (G, (\chi_1)_A \times 1_B) \geq (NB,(\chi_1)_N \times 1_B)$$
$$=(NB,\theta \times 1_B)=(NB,(\chi_2)_N \times 1_B) \leq (G, (\chi_2)_A \times 1_B)
\geq (A,(\chi_2)_A) \leq (G,\chi_2).$$
Hence $(K_1,\psi_1) \leq (G,\chi_1)\sim(G,\chi_2) \geq (K_1,\psi_2)$ and $f^{-1}(N, \theta)$ is connected.
It follows from Lemma~\ref{lem-component} that 
$$\pi_0\Gamma_{p,1}(G)=\bigcup_{\theta \in \Irr(N)} \pi_0f^{-1}(N,\theta)$$
and $|\pi_0\Gamma_{p,1}(G)|=p$, as desired.

Now assume that $G$ is non-Abelian. Let $\Irr(N)=\{\theta_0, \theta_1,...,\theta_{p-1}\}$ with $\theta_0=1_N$. It follows from Lemma~\ref{lem-non-abelian group of order p^3} that for each $1 \leq i \leq p-1$, there exists $\chi_i \in \Irr(G)$ such that $\theta_i^G=p\chi_i$ and $(\chi_i)_N=p\theta_i$. 
Write $\chi_0=1_G$.  We can define a map $g: \Gamma_p(N) \rightarrow\Gamma_{p,1}(G)$ with 
$$g(N,\theta_i)=(G,\chi_i)~\text{for}~0 \leq i \leq p-1.$$
Clearly $fg=\Id_{\Gamma_{p,1}(N)}.$ 
On the other hand, for each $(K,\psi) \in \Gamma_{p,1}(G)$,  assume that $\psi_N=\lambda \theta_i$ for some $i$. Then
$$gf(K,\psi)=g(N,\theta_i)=(G,\chi_i).$$
Since $\psi$ is a constitute of $\theta_i^K$, $\psi^G$ is also a constituent of $(\theta_i^K)^G=(\theta_i)^G$.
Note that, by Lemma~\ref{lem-non-abelian group of order p^3}, $\chi_i$ is the unique irreducible constituent of $(\theta_i)^G$. Hence $\chi_i$ is an irreducible constituent of $\psi^G$, which implies that $(K,\psi) \leq (G,\chi_i)$.
Hence we have
$$gf(K,\psi)=(G,\chi_i)\geq (K,\psi),$$
and $gf \simeq \Id_{\Gamma_{p,1}(G)}$.
Therefore $\Gamma_{p,1}(G)$ is homotopy-equivalent to $\Gamma_p(N)$, which implies that
$|\pi_0\Gamma_{p,1}(G)|=|\pi_0\Gamma_p(N)|=p$.
\end{proof}
Let $G=HK$ be the semidirect product of two subgroups $H$ and $K$ with $H \unlhd G$.  For each $\varphi \in \Irr(K)$, the lift of $\varphi$ to $G$ is defined by 
$$\bar{\varphi}: G \rightarrow \mathbb{C};~hk\mapsto \varphi(k).$$
Clearly, $\bar{\varphi}$ is an irreducible character of $G$ with $\bar{\varphi}_K=\varphi$ and $\bar{\varphi}_H=\varphi(1)1_H$.

\begin{lem}\label{lem-semi-direct-p^2}
Let $G=HK$ be the semidirect product of two subgroups $H$ and $K$ with $H \unlhd G$. 
Suppose that $|H|=|K|=p^2$ for some prime $p$. Then $\Gamma_{p,1}(G)$ is connected.
\end{lem}
\begin{proof}
Assume that $\Gamma_{p,1}(G)$ is disconnected. Applying Theorem~\ref{thm-A} for $e=1$, $G$ has no subgroups isomorphic to $\C_p \times \C_p \times \C_p$ or $\C_{p^2} \times \C_{p^2}$.

If $K \cong \C_p \times \C_p$, by Lemma~\ref{basis of p-subgroup}(1), we can choose $1 \neq Z \leq H \cap \Z(G)$ with order $p$ and $G$ has a subgroup $ZK \cong \C_p \times \C_p \times \C_p$, which is a contradiction. If $H \cong \C_p \times \C_p$, as $H \unlhd G$,  it implies that 
$K/\C_K(H)$ is isomorphic to a $p$-subgroup $\Aut(H)$. As $|\Aut(H)|_p=|\GL(2,p)|_p=p$,  $|K/\C_K(H)|$ divides $p$. Hence $|\C_K(H)| \geq p$ as $|K|=p^2$. Hence $\C_K(H)H$ has a subgroup isomorphic to $\C_p \times \C_p \times \C_p$, which is also a contradiction.

Hence $H$ and $K$ are both isomorphic to $\C_{p^2}$. For each $\varphi\in \Irr(K)$, 
let $\bar{\varphi}: G \rightarrow \mathbb{C}$ with $\bar{\varphi}(hk)=\varphi(k)$ for $h \in H, k\in K$, which is the lift of $\varphi$ to $G$. It is easy to see that $\bar{\varphi}_K=\varphi$ and
$\bar{\varphi}_H=\varphi(1) 1_H$. Then there exists a chain in $\Gamma_{p,1}(G):$
$$(K,1_K)\leq (HK,1_G) \geq (H,1_H)\leq (HK, \bar{\varphi}) \geq (K,\varphi),$$ which implies that $[(K,1_K)]=[(K,\varphi)]$ for each $\varphi \in \Irr(K)$. Then it follows from Lemma~\ref{lem-subgroup-connected} that 
$$\pi_0\Gamma_{p,1}(G)=\{[(K,\varphi)] \mid \varphi \in \Irr(K)\}=\{[(K,1_K)]\}.$$
Hence $\Gamma_{p,1}(G)$ is connected, which is the final contradiction.
\end{proof}

\begin{proof}[\textbf{\emph{Proof of Theorem~\ref{thm-C}}}]
Let $\mathscr{X}$ be the set of all subgroups of $G$ with order $p^2$ and $I=\bigcap_{H \in \mathscr{X}} H$ by definition. Clearly $|I| \leq p^2$.  Note that if $|I|=p^2$, then $G$ has only one subgroup with order $p^2$. It follows from Lemma~\ref{basis of p-subgroup}(3) and Lemma~\ref{lem-p^2 components} that $G$ has exactly $p^2$ connected components, as desired. Hence we may only show the case that $I=1$ or $|I|=p$.

We first assume that $I=1$ and we will prove that $\Gamma_{p,1}(G)$ is connected.
In this case, we have that $|G|\geq p^3$. By Lemma~\ref{basis of p-subgroup}, we take can a normal subgroup $H$ of $G$ with order $p^2$ and $\Z(G) \cap H \neq 1$. Choose a subgroup $1\neq Z\leq \Z(G) \cap H$ with $|Z|=p$. 

If there exists a subgroup $K \in \mathscr{X}$ such that $H \cap K=1$, then $HK$ is a semidirect product of two subgroups of order $p^2$. It follows from Lemma~\ref{lem-semi-direct-p^2} that
$\Gamma_{p,1}(HK)$ is connected. By Lemma~\ref{lem-subgroup-connected}, $\Gamma_{p,1}(G)$ is connected, as desired.

Now we assume that $H\cap K \neq 1$ for every $K\in \mathscr{X}$. It implies that
$|H\cap K|=p$ for all $K\in \mathscr{X}-\{H\}$. Note that
$$1=I= \bigcap_{K \in \mathscr{X}-\{H\}} H \cap K.$$ 
Then there exist $K_1$ and $K_2$ in $ \mathscr{X}-\{H\}$ such that 
$(K_1\cap H)\cap (K_2\cap H)=1.$
Write $H_1=K_1\cap H$ and $H_2=K_2\cap H$. Then $H_1\cap H_2=1$ and $H=H_1\times H_2$. 
Since $Z \leq H$, either $Z\cap H_1=1$ or $Z\cap H_2=1$ holds.
Without loss of generality, we assume that $Z\cap H_1=1$ and so $H=H_1\times Z$. For each $\varphi \in \Irr(H)$ and $\varphi_{H_1}\in \Irr(H_1)$, there exists a chain
$$(H,1_H)=(H_1 \times H_2,1_{H_1}\times 1_{H_2})\leq (K_2H_1, 1_{K_2}\times 1_{H_1})\geq (K_2,1_{K_2})$$
$$\leq (K_2 H_1,1_{K_2}\times \varphi_{H_1}) \geq (H_2 \times H_1,1_{H_2}\times \varphi_{H_1})=(H,\chi) ~\text{in}~\Gamma_{p,1}(G),$$
where $\chi=1_{H_2}\times \varphi_{H_1}$ and clearly $\chi_{H_1}=\varphi_{H_1}$. Let $\psi \in \Irr(K_1)$ and $\varphi_{H_1}=\psi_{H_1}$ as $K_1$ is Abelian. As $Z\cap K_1=Z\cap K_1 \cap H=1$, so $H K_1=Z\times K_1$. Hence
$$(H,\chi)=(Z\times H_1, \chi_Z \times \chi_{H_1} )=(Z\times H_1, \chi_Z \times \varphi_{H_1} )\leq (Z\times K_1,\chi_Z\times \psi)\geq (K_1,\psi) $$
$$\leq (Z\times K_1, \varphi_Z\times \psi) \geq (Z\times H_1,\varphi_Z\times \varphi_{H_1})=(H,\varphi)~\text{in}~\Gamma_{p,1}(G),$$
which implies $(H,\varphi)\sim (H,\chi) \sim (H,1_H)$ for each $\varphi \in \Irr(H)$ and consequently $\Gamma_{p,1}(H)$ is connected. By Lemma~\ref{lem-subgroup-connected}, $\Gamma_{p,1}(G)$ is connected.

Now we assume that $|I|=p$.
For each $(H,\varphi) \in \Gamma_{p,1}(G)$, as $|H|\geq p^2$, $I \leq H$.
Note that $I \in \Z(G)$ as $I \unlhd G$. Then $\varphi_I=\lambda \theta$ for some integer $\lambda \geq 1$ and some $\theta \in \Irr(I)$. We can consider the map 
$$\alpha: \Gamma_{p,1}(G) \rightarrow \Gamma_p(I);~ (H,\varphi) \mapsto (I, \theta),$$
where $\theta$ is the unique irreducible constitute of $\varphi_I$. It is not difficult to check that $\alpha$ is a surjective map.
Note that $\Gamma_p(I)=\{(I,\theta) \mid \theta \in \Irr(I)\}$ is a discrete space of points $p$, hence it has $p$ connected components. By Lemma~\ref{lem-component}, $|\pi_0\Gamma_{p,1}(G)| \geq |\pi_0\Gamma_p(I)|=p$.

Since $|I|=p$, $G$ has at least two distinct subgroups with order $p^2$.
Let $H,K \in \mathscr{X}$ such that $H \unlhd G$ and $H \cap K=I$.
Write $T=HK$ and $T$ is a group of order $p^3$ by an order argument. Note that $H,K$ are maximal in $T$ and so $\Phi(T) \leq H \cap K=I$.

We claim that $\Phi(T)=I$ is of order $p$. Otherwise, $\Phi(T)=1$ and $T \cong \C_p \times \C_p \times \C_p$. In this case, $T$ has three subgroups with order $p^2$ such that their intersection is trivial, contrary to $|I|=p$. Hence $|\Phi(T)|=p$ and 
Lemma~\ref{Frattini subgroup of order p} shows that $\Gamma_{p,1}(T)$ has exactly $p$ connected components. It follows from
Lemma~\ref{lem-subgroup-connected} that $|\pi_0\Gamma_{p,1}(G)| \leq |\pi_0\Gamma_{p,1}(T)|=p$.
Hence we also get that 
$|\Gamma_{p,1}(G)|=p=|I|$ and
the proof is complete.
\end{proof}

\section{Further remarks and questions}
As mentioned in Theorem~\ref{thm-C}, for a $p$-group $G$, $\Gamma_{p,1}(G)$ is disconnected if and only if the intersection $I$ of all subgroups of $G$ with order $p^2$ is non-trivial. It would be natural to ask the following question:
\begin{question}\label{Q-1}
Determine all $p$-groups $G$ such that $I \neq 1$, where $I$ is the intersection of all subgroups of $G$ with order $p^2$.
\end{question}
As shown in the proof, $G$ has no subgroups isomorphic to $\C_p \times \C_p \times \C_p$ and $\C_{p^2} \times \C_{p^2}$. Hence $G$ is a $p$-rank at most two, which is classified by Blackburn
(See~\cite[Theorem 4.1]{Blackburn1961}) if $p$ is odd and we know little for the case $p=2$.

Motivated by Theorem~\ref{thm-B} ($e=0$) and Theorem~\ref{thm-C} ($e=1$), we propose another interesting question for the case that $e \geq 2$:
\begin{question}\label{Q-2}
Let $G$ be a $p$-group for a prime $p$. Let $e$ be an integer $\geq 2$ and let $I$ be the intersection of all subgroups of $G$ with order $p^{e+1}$. 
Is it true that $\Gamma_{p,e}(G)$ is disconnected if and only if $I \neq 1$? 
\end{question}

Note that $I \unlhd G$. Suppose that $I \neq 1$ and let $Z \leq I \cap \Z(G) \neq 1$ with $|Z|=p$. For each $(K, \psi) \in \Gamma_{p,e}(G)$, $Z \leq I \leq K$ and $\psi_Z$ has the unique irreducible constituent $\theta \in \Irr(Z)$. Hence there is a surjective poset map $\alpha: \Gamma_{p,e}(G) \rightarrow \Gamma_p(Z)$ with $\alpha(K,\psi)=(Z,\theta)$. By Lemma~\ref{lem-component}, $|\pi_0\Gamma_{p,e}(G)| \leq |\Gamma_p(Z)|=p$, which implies that $\Gamma_{p,e}(G)$ is disconnected. Thus it will be interesting to prove the converse direction of Question~\ref{Q-2}.

\bibliographystyle{unsrt}
  \bibliography{bibMHY}

\end{document}